\documentclass{amsart}
\usepackage{amscd,amssymb,latexsym}
\begin{document}
\def\a{{\alpha}}	\def\b{{\beta}}	\def\l{{\lambda}} 	\def\la{{\langle}}	
\def\pa{{\partial}} 	\def\ra{{\rangle}}	\def\To{{\longrightarrow}} 	 

\def\AA{{\mathbb A}}	\def\CC{{\mathbb C}} 	\def\FF{{\mathbb F}}
\def\HH{{\mathbb H}}
\def\KK{{\mathbb K}}	\def\NN{{\mathbb N}}	\def\PP{{\mathbb P}} 
\def\QQ{{\mathbb Q}}	\def\RR{{\mathbb R}}  	\def\ZZ{{\mathbb Z}}

\def\charr{\operatorname {char}}\def\codim{\operatorname {codim}}	
\def\GL{\operatorname {GL}}
\def\Hom{\operatorname {Hom}}
\def\Im{\operatorname {Im}}	\def\Ind{\operatorname {Ind}}
\def\Ker{\operatorname {Ker}}  	
\def\rad{\operatorname {rad}}	\def\rank{\operatorname {rank}}   
\def\reg{\operatorname {reg}}	\def\Rep{\operatorname {Rep}} 	   
\def\Res{\operatorname {Res}} 	      
\def\Ver{\operatorname {Ver}} 
\def\Sym{\operatorname {Sym}}   
\def\Tor{\operatorname {Tor}} 	

\newtheorem{Theorem}{Theorem}[section]
\newtheorem{Corollary}[Theorem]{Corollary}
\newtheorem{Lemma}[Theorem]{Lemma}
\newtheorem{Proposition}[Theorem]{Proposition}
\newtheorem{Question}[Theorem]{Question}
\newtheorem{Remark}[Theorem]{Remark}
\newtheorem{Conjecture}{Conjecture}
\newtheorem{df}[Conjecture]{Definition}
\numberwithin{equation}{section}

\title{$N_6$ property for third Veronese embeddings}  
\author{Thanh Vu}
\address{Department of Mathematics, University of California at Berkeley, Berkeley, CA 94720}
\email{vqthanh@math.berkeley.edu}
\subjclass[2010]{Primary 13D02, 14M12, 05E10}
\keywords{Syzygies, Veronese varieties, matching complexes.}
\date{February 28, 2013}

\begin{abstract} The rational homology groups of the matching complexes are closely related to the syzygies of the Veronese embeddings. In this paper we will prove the vanishing of certain rational homology groups of matching complexes, thus proving that the third Veronese embeddings satisfy the property $N_6$. This settles the Ottaviani-Paoletti conjecture for third Veronese embeddings. This result is optimal since $\nu_3(\PP^n)$ does not satisfy the property $N_7$ for $n\ge 2$ as shown by Ottaviani-Paoletti in \cite{OP}.
\end{abstract}
\maketitle

\section{Introduction}
Let $k$ be a field of characteristic $0$. Let $V$ be a finite dimensional vector space over $k$ of dimension $n+1$. The projective space $\PP(V)$ has coordinate ring naturally isomorphic to $\Sym V$. For each natural number $d$, the $d$-th Veronese embedding of $\PP(V)$, which is naturally embedded into the projective space $\PP(\Sym^dV)$ has coordinate ring $\Ver(V,d) = \oplus_{k=0}^\infty \Sym^{kd} V$. For each set of integers $p, q, b$, let $K^d_{p,q}(V,b)$ be the associated Koszul cohomology group defined as the homology of the $3$ term complex
\begin{align*}
\bigwedge^{p+1} \Sym^dV & \otimes \Sym^{(q-1)d+b}V \to \bigwedge^p \Sym^dV \otimes \Sym^{qd + b} V \\
&\to \bigwedge^{p-1}\Sym^dV \otimes \Sym^{(q-1)d+b}V.
\end{align*}
Then $K_{p,q}^d(V,b)$ is the space of minimal $p$-th syzygies of degree $p+q$ of the $\GL(V)$ module $\oplus_{k=0}^\infty \Sym^{kd +b}V.$ We write $K_{p,q}^d(b): {\operatorname{Vect}} \to {\operatorname{Vect}}$ for the functor on finite dimensional $k$-vector spaces that assigns to a vector space $V$ the corresponding syzygy module $K_{p,q}^d(V,b)$. 

\begin{Conjecture}[Ottaviani-Paoletti] $$K_{p,q}^d(V,0) = 0 \text{ for } q \ge 2  \text{ and } p \le 3d - 3.$$ 
\end{Conjecture}

The conjecture is known for $d = 2$ by the work of Josefiak-Pragacz-Weyman \cite{JPW} and also known for $\dim V = 2$ and $\dim V = 3$ by the work of Green \cite{G} and Birkenhake \cite{B}. All other cases are open. Also, Ottaviani and Paoletti in \cite{OP} showed that $K_{p,2}^d(V,0) \neq 0$ for $p = 3d - 2$, when $\dim V \ge 3$, $d\ge 3$. In other words, the conjecture is sharp. Recently, Bruns, Conca, and R\"omer in \cite{BCR} showed that $K_{d+1,q}^d(V,0) = 0$ for $q \ge 2$. In this paper we will prove the conjecture in the case $d = 3$. Thus the main theorem of the paper is:

\noindent 
{\bf Theorem \ref{main}.} The third Veronese embeddings of projective spaces satisfy property $N_6$.

To prove Theorem \ref{main} we prove that vanishing results hold for certain rational homology groups of matching complexes (defined below), and then use \cite[Theorem~5.3]{KRW} to translate between the zyzygy modules $K_{p,q}^d(V,b)$ and the homology groups of matching complexes.

\begin{df}[Matching Complexes] Let $d > 1$ be a positive integer and $A$ a finite set. The matching complex $C^d_A$ is the simplicial complex whose vertices are all the $d$-element subsets of $A$ and whose faces are $\{A_1,..., A_r\}$ so that $A_1, ..., A_r$ are mutually disjoint. 
\end{df}

The symmetric group $S_A$ acts on $C^d_A$ by permuting the elements of $A$ making the homology groups of $C^d_A$ representations of $S_A$. For each partition $\l$, we denote by $V^\l$ the irreducible representation of $S_{|\l|}$ corresponding to the partition $\l$, and $S_\l$ the Schur functor corresponding to the partition $\l$. For each vector space $V$, $S_\l(V)$ is an irreducible representation of $\GL(V)$. The relation between the syzygies of the Veronese embeddings and the homologies of matching complexes is given by the following theorem of Karaguezian, Reiner and Wachs.

\begin{Theorem}\cite[Theorem~5.3]{KRW}\label{trans} Let $p,q$ be non-negative integers, let $d$ be a positive integer and let $b$ be a non-negative integer. Write $N = (p+q)d + b$. Consider a partition $\l$ of $N$. Then the multiplicity of $S_\l$ in $K_{p,q}^d(b)$ coincides with the multiplicity of the irreducible $S_N$ representation $V^\l$ in $\tilde H_{p-1}(C^d_N)$.
\end{Theorem} 
This correspondence makes Conjecture $1$ equivalent to the following conjecture.

\begin{Conjecture} The only non-zero homology groups of $C^d_{nd}$ for $n = 1, ..., 3d - 1$ is $\tilde H_{n-2}$.
\end{Conjecture}
We will prove this conjecture for $d = 3$ by computing the homology groups of $C^3_n$ by induction. To compute the homology groups of the matching complexes inductively, the following equivariant long-exact sequence introduced by Raicu in \cite{R} is useful. Let $A$ be a finite set with $|A| \ge 2d$. Let $a \in A$ be an element of $A$. Let $\a$ be a $d$-element subset of $A$ such that $a\in \a$. Let $\b = \a \setminus a$, and let $C = A \setminus \a$, $B = A \setminus \{a\}$. Then we have the following long-exact sequence of representations of $S_B$. 

\begin{align}
  \label{equiv}
  \begin{split}
    \cdots & \to \Ind_{S_C\times S_\b}^{S_B} (\tilde H_r(C^d_C)\otimes 1) \to \tilde H_r(C^d_B)\to \Res_{S_A}^{S_B}(\tilde H_r(C^d_A))\\
    &\to \Ind_{S_C \times S_\b}^{S_B} (\tilde H_{r-1}(C^d_C) \otimes 1) \to \cdots
  \end{split}
\end{align}

Morever, for each $b$, $0\le b\le d-1$, the $\GL(V)$ module $\oplus_{k=0}^\infty \Sym^{kd + b}V$ is a Cohen-Macaulay module in the coordinate ring of the projective space $\PP(\Sym^dV)$. The dual of the resolution of each of these modules is the resolution of another such module giving us the duality among the Koszul homology groups $K_{p,q}^d(V,b)$. For simplicity, since we deal with third Veronese embeddings, we will from now on assume that $d = 3$. To compute the homology groups of the matching complexes $C^3_n$ for $n \le 10$ we use the duality in the case $\dim V = 2$, and to compute the homology groups of the matching complexes $C^3_n$ for $n \ge 14$, we use the duality in the case $\dim V = 3$, so we will make them explicit here. In the case $\dim V  =2$, the canonical module of $\Ver(V,3)$ as representation of $\GL(V)$ is $S_\l(V)$ with $\l = (5,5)$. From \cite[Chapter~2]{W},
$$\Hom(S_{(a,b)},S_{(5,5)}) \cong S_{(5-b,5-a)}$$
and the correspondence in Theorem \ref{trans}, we have 
\begin{Proposition}\label{duality2} The multiplicity of $V^\l$ with $\l = (\l_1,\l_2)$ in the homology group $\tilde H_{p-1}(C^3_N)$ coincides with the multiplicity of $V^\mu$ with $\mu = (5-\l_2,5-\l_1)$ in the homology group $\tilde H_{1-p}(C^3_{10-N}).$
\end{Proposition}
Similarly, in the case $\dim V = 3$, the canonical module of $\Ver (V,3)$ as representation of $\GL(V)$ is $S_\l(V)$ with $\l = (9,9,9)$. From \cite[Chapter~2]{W},
$$\Hom(S_{(a,b,c)},S_{(9,9,9)}) \cong S_{(9-c,9-b,9-a)}$$
and the correspondence in Theorem \ref{trans}, we have 
\begin{Proposition}\label{duality3} The multiplicity of $V^\l$ with $\l = (\l_1,\l_2,\l_3)$ in the homology group $\tilde H_{p-1}(C^3_{N})$ coincides with the multiplicity of $V^\mu$ with $\mu = (9-\l_3,9-\l_2,9-\l_1)$ in the homology group $\tilde H_{6-p}(C^3_{27-N}).$
\end{Proposition}

Finally, to determine the homology groups of the matching complex $C^3_N$, we apply the equivariant long exact sequence (\ref{equiv}) to derive equalities and inequalities for the multiplicities of the possible irreducible representations in the unknown homology groups of $C^3_N$. This is carried out with the help of our Macaulay2 package MatchingComplex.m2 that we will explain in the appendix. We then use Maple to solve this system of equalities and inequalities. 

The paper is organized as follows. In the second section, using the equivariant long exact sequence (\ref{equiv}) and the duality in Proposition \ref{duality2} we compute homology groups of matching complexes $C^3_n$ for $n \le 13$. In the third section, using the results in the second section and the duality in Proposition \ref{duality3} we derive all irreducible representations whose corresponding partitions have at most $3$ rows in the homology groups of $C^3_n$ for $n \ge 14$. The systems of equalities and inequalities from exact sequences obtained by applying the equivariant long exact sequence (\ref{equiv}) with $|A| = 20$ and $|A| = 23$ will not determine $\tilde H_4(C^3_{20})$ and $\tilde H_5(C^3_{23})$ uniquely. Thus we need to compute the dimensions of $K_{p,1}(V,2)$ with $p = 5, 6$ and $\dim V = 4$. This is done by using Macaulay2 to compute the dimensions of the spaces of minimal $p$-th syzygies of degree $p+1$ for $p = 6, 7$ of the module $\oplus_{k=0}^\infty \Sym^{3k + 2}V$ with $\dim V = 4$. We then state the results for the homology groups of $C^3_n$ with $14\le n \le 24$ and finish the proof of our main theorem. The proof of Proposition \ref{C23} illustrates the computation of the homology groups of the matching complexes $C^3_n$ dealing with the most complicated matching complex $C^3_{23}$ in the series. In the appendix we explain the ideas behind our package leading to the computation of the homology groups of the matching complexes.

\section{Homology of matching complexes}\label{Sec1}

In this section, using the equivariant long exact sequence (\ref{equiv}) we compute the homology groups of the matching complexes $C^3_n$ for $n \le 13$. In the following, we denote the partition $\l$ with row lengths $\l_1 \ge \l_2 \ge ... \ge \l_k \ge 0$ by the sequence $(\l_1, \l_2,...,\l_k)$ and we use the same notation for the representation $V^\l$. To simplify notation, we omit the subscript and superscript when we use the operators $\Ind$ and $\Res$. It is clear from the context and the equivariant long exact sequence what the induction and restriction are. From the definition of the matching complexes, it is not hard to see the following. 

\begin{Proposition}\label{Prop1.1} The only non-vanishing homology groups of $C^3_n$ with $n = 4, 5, 6$ are respectively
$$\tilde H_0 C^3_4 = (3,1),\;\; \tilde H_0 C^3_5 = (4,1) \oplus (3,2), \;\; \tilde H_0 C^3_6 = (4,2).$$
\end{Proposition}

Together with Proposition \ref{duality2} we have 
\begin{Proposition}\label{2row} The only irreducible representations whose corresponding partitions have at most $2$ rows in the homology groups of the matching complexes $C^3_n$, $7\le n\le 10$ are 
$$(\tilde H_1C^3_8)_2 = (5,3),\;\; (\tilde H_1C_9^3)_2 = (5,4),\;\; (\tilde H_1C_{10}^3)_2 = (5,5)$$
where $(\tilde H_iC^3_n)_2$ denote the subrepresentation of $\tilde H_iC^3_n$ consisting of all irreducible representations of $\tilde H_iC^3_n$ whose corresponding partitions have at most $2$ rows.
\end{Proposition}  

\begin{Proposition} The only non-vanishing homology group of $C_7^3$ is
$$\tilde H_1 C^3_7 = (5,1,1) \oplus (3,3,1).$$
\end{Proposition}
\begin{proof} Applying the equivariant long exact sequence (\ref{equiv}) with $|A| = 7$ we have exact sequences 
$$0\to \Res \tilde H_i C_7^3 \to 0$$
for $i\neq 0,1$, and an exact sequence
$$0 \to \Res \tilde H_1 C_7^3 \to \Ind \tilde H_0 C_4^3 \to \tilde H_0 C_6^3 \to \Res \tilde H_0 C_7^3 \to 0.$$
Therefore, $\tilde H_i C_7^3 = 0$ for $i\neq 0, 1$. To show that $\tilde H_0C_7^3$ is zero, note that $\tilde H_0C_6^3$ maps surjectively onto $\Res \tilde H_0 C_7^3$. Since $\tilde H_0C_6^3 = (4,2)$ and by Proposition \ref{2row}, $\tilde H_0C_7^3$ does not contain any irreducible representations whose corresponding partitions have at most $2$ rows, it must be zero. Therefore, we know that $\Res \tilde H_1C_7^3$ as representation of $S_6$ is equal to 
$$\Ind \tilde H_0 C_4^3 - \tilde H_0 C_6^3 = (5,1) \oplus (4,1,1)\oplus (3,3)\oplus (3,2,1).$$
Moreover, by Proposition \ref{2row}, $\tilde H_1C_7^3$ does not contain any irreducible representations whose corresponding partitions have at most $2$ rows, thus $\tilde H_1C_7^3$ must consist $(5,1,1)$ and $(3,3,1)$ as its restriction contains $(5,1)$ and $(3,3)$. But the restriction of $(5,1,1) \oplus (3,3,1)$ is equal to $\Res \tilde H_1C_7^3$ so we have the desired conclusion.
\end{proof} 

\begin{Proposition} The only non-vanishing homology group of $C_8^3$ is 
$$\tilde H_1 C^3_8 = (6,1,1) \oplus (5,2,1) \oplus (5,3) \oplus (4,3,1) \oplus (3,3,2).$$
\end{Proposition} 
\begin{proof} Applying the equivariant long exact sequence (\ref{equiv}) with $|A| = 8$ we have exact sequences 
$$0 \to \Res \tilde H_i C_8^3 \to 0$$
for $i\neq 1$, and an exact sequence
$$0 \to \tilde H_1 C_7^3 \to \Res  \tilde H_1 C_8^3 \to \Ind \tilde H_0 C_5^3 \to 0.$$
Therefore, $\tilde H_i C_8^3 = 0$ for $i\neq 1$. To compute $\tilde H_1C_8^3$, note that from the exact sequence 
$$\Res \tilde H_1C_8^3 = \tilde H_1 C_7^3 + \Ind \tilde H_0 C_5^3$$ 
which consists of irreducible representations whose corresponding partitions have at most $3$ rows, therefore $\tilde H_1C_8^3$ contains only irreducible representations whose corresponding partitions have at most $3$ rows. Moreover, by Proposition \ref{2row}, the only irreducible representation whose corresponding partition has at most $2$ rows in $\tilde H_1C_8^3$ is $(5,3)$. Therefore, the restrictions of irreducible representations whose corresponding partitions have $3$ rows in $\tilde H_1C_8^3$ is equal to 
$$(3,2,2) \oplus 2 \cdot (3,3,1) \oplus 2\cdot (4,2,1) \oplus (4,3) \oplus 2\cdot(5,1,1) \oplus (5,2) \oplus (6,1).$$
It is easy to see that the set of irreducible representations whose corresponding partitions have $3$ rows in $\tilde H_1C_8^3$ has to be equal to $(6,1,1) \oplus (5,2,1) \oplus (4,3,1) \oplus (3,3,2)$.
\end{proof} 

\begin{Proposition} The only non-vanishing homology group of $C_9^3$ is 
$$\tilde H_1 C^3_9 = (6,2,1) \oplus (5,4) \oplus (5,3,1) \oplus (4,3,2).$$
\end{Proposition} 
\begin{proof}
Applying the equivariant long exact sequence (\ref{equiv}) with $|A| = 9$ we have exact sequences
$$0\to \Res \tilde H_i C_9^3 \to 0$$
for $i\neq 1$, and an exact sequence
$$0 \to \tilde H_1 C_8^3 \to \Res \tilde H_1 C_9^3 \to \Ind \tilde H_0 C_6^3  \to 0.$$
Therefore, $\tilde H_i C_9^3 = 0$ for $i\neq 1$, and 
$$\Res \tilde H_1 C_9^3 = \tilde H_1C_8^3 + \Ind \tilde H_0C_6^3.$$
Moreover, by Proposition \ref{2row}, the only irreducible representation whose corresponding partition has at most $2$ rows in $\tilde H_1 C_9^3$ is $(5,4)$. Therefore, the restrictions of irreducible representations whose corresponding partitions have 3 rows in $\tilde H_1C_9^3$ is equal to 
$$(3,3,2) \oplus (4,2,2) \oplus 2\cdot (4,3,1) \oplus 2\cdot (5,2,1) \oplus (5,3) \oplus (6,1,1) \oplus (6,2).$$
It is easy to see that the set of irreducible representations whose corresponding partitions have $3$ rows in $\tilde H_1C_9^3$ has to be equal to $(6,2,1) \oplus (5,3,1) \oplus (4,3,2)$.
\end{proof}

\begin{Proposition} The only non-vanishing homology groups of $C_{10}^3$ are
$$\tilde H_2 C^3_{10} = (7,1,1,1) \oplus (5,3,1,1) \oplus (3,3,3,1),\;\; \tilde H_1 C^3_{10} = (5,5).$$
\end{Proposition} 
\begin{proof} Applying the equivariant long exact sequence (\ref{equiv}) with $|A| = 10$ we have exact sequences 
$$0\to \Res \tilde H_iC_{10}^3 \to 0$$
for $i\neq 1,2$, and an exact sequence
$$0 \to \Res \tilde H_2C_{10}^3 \to \Ind \tilde H_1 C_7^3 \to \tilde H_1 C_9^3 \to \Res \tilde H_1 C_{10}^3 \to 0.$$
Therefore, $\tilde H_iC_{10}^3 = 0$ for $i\neq 1,2$. Since $\tilde H_1C_9^3$ maps surjectively onto $\Res \tilde H_1C_{10}^3$, and by Proposition \ref{2row}, $\tilde H_1C_{10}^3$ contains the irreducible representation $(5,5)$, $(6,2,1) \oplus (5,3,1) \oplus (4,3,2)$ maps surjectively onto the restrictions of the other irreducible representations of $\tilde H_1C_{10}^3$. Moreover, there is no irreducible representation of $S_{10}$ whose restriction is a subrepresentation of $(6,2,1) \oplus (5,3,1) \oplus (4,3,2)$, thus $\tilde H_1C_{10}^3 = (5,5)$. Therefore, $\Res \tilde H_2C_{10}^3$ is equal to 
\begin{align*}
\Ind \tilde H_1 C_7^3 + \Res \tilde H_1 C_{10}^3 &- \tilde H_1C_9^3 = (3,3,2,1)\oplus(3,3,3)\oplus (4,3,1,1)\\
&\oplus(5,2,1,1)\oplus(5,3,1)\oplus(6,1,1,1)\oplus(7,1,1).
\end{align*}
Moreover, since $\tilde H_1C_{10}^3  = (5,5)$, by Theorem \ref{trans}, $K_{2,1}(V,1) = S_\l(V)$ with $\l = (5,5)$. Since $\oplus_{k=0}^\infty \Sym^{3k+1}(V)$ with $\dim V = 3$ is a Cohen-Macaulay module of codimension $7$ with $h$-vector $(3,6)$, 
$$\dim K_{3,0}(V,1) = 6\cdot \binom{7}{2}+ \dim K_{2,1}(V,1) - 3 \cdot \binom{7}{3} = 0.$$
By Theorem \ref{trans}, $\tilde H_2C_{10}^3$ does not contain any irreducible representations whose corresponding partitions have at most $3$ rows, therefore it must contain irreducible representations $(3,3,3,1)$, $(5,3,1,1)$, $(7,1,1,1)$ each with multiplicity $1$. But the sum of the restrictions of these irreducible representations is equal to the restriction of $\tilde H_2C_{10}^3$ already, thus we have the desired conclusion.
\end{proof}

\begin{Proposition} The only non-vanishing homology group of $C_{11}^3$ is
\begin{align*}
\tilde H_2 C^3_{11} & = (7,3,1) \oplus(6,4,1) \oplus(6,3,2)\oplus (5,4,2) \oplus(5,3,3)\\
& \oplus (6,3,1,1) \oplus (8,1,1,1) \oplus (7,2,1,1) \oplus(5,4,1,1)\\
& \oplus(5,3,2,1)\oplus(4,3,3,1)\oplus(3,3,3,2).
\end{align*}
\end{Proposition} 
\begin{proof}Applying the equivariant long exact sequence (\ref{equiv}) with $|A| = 11$ we have exact sequences 
$$0\to \Res \tilde H_i C_{11}^3 \to 0$$
for $i\neq 1, 2$ and an exact sequence
$$0 \to \tilde H_2 C_{10}^3 \to \Res \tilde H_2 C_{11}^3  \to \Ind \tilde H_1 C_8^3  \to \tilde H_1 C_{10}^3 \to \Res \tilde H_1 C_{11}^3 \to 0.$$
Therefore, $\tilde H_i C_{11}^3 = 0$ for $i\neq 1,2$. To show that $\tilde H_1C_{11}^3$ is zero, note that $\tilde H_1 C_{10}^3$, consists of partition $(5,5)$ only, maps surjectively onto $\Res \tilde H_1 C_{11}^3$, so $\tilde H_1C_{11}^3 = 0$. Thus 
$$\Res \tilde H_2C_{11}^3  = \tilde H_2C_{10}^3 + \Ind \tilde H_1C_8^3 - \tilde H_1C_{10}^3.$$
Moreover, $\tilde H_2C_{11}^3$ does not contain any irreducible representations whose corresponding partitions have at most $2$ rows, but $\Res \tilde H_2C_{11}^3$ contains $(6,4)$ and $(7,3)$, thus $\tilde H_2C_{11}^3$ must contain $(6,4,1)$ and $(7,3,1)$ each with multiplicity $1$. The restrictions of the remaining irreducible representations is equal to
$$A = \tilde H_2C_{10}^3 + \Ind \tilde H_1 C_8^3 - \tilde H_1 C_{10}^3 - \Res (6,4,1) - \Res (7,3,1).$$
To find a set of irreducible representations whose sum of the restrictions is $A$, we first induce $A$. We then eliminate all irreducible representations whose restrictions are not contained in $A$. After that, we have a set of partitions $B$. Then we write down the set of equations that make the restriction of $B$ equal to $A$. Then we solve for the non-negative integer solutions of that system. For this problem, it is easy to see that we have a unique solution as in the statement of the proposition.
\end{proof}

\begin{Proposition} The only non-vanishing homology group of $C_{12}^3$ is
\begin{align*}
\tilde H_2 C^3_{12} & = (7,4,1) \oplus(7,3,2) \oplus (6,5,1) \oplus(6,4,2) \oplus(6,3,3)\oplus(5,5,2)\\
&\oplus(5,4,3)\oplus (8,2,1,1) \oplus(7,3,1,1)  \oplus(6,4,1,1)  \oplus(6,3,2,1)\\
&\oplus(5,4,2,1)\oplus(5,3,3,1)\oplus(4,3,3,2).
\end{align*}
\end{Proposition} 
\begin{proof}Applying the equivariant long exact sequence (\ref{equiv}) with $|A| = 12$ we have exact sequences 
$$0 \to \Res \tilde H_i C_{12}^3 \to 0$$
for $i\neq 2$ and an exact sequence
$$0 \to \tilde H_2C_{11}^3 \to \Res \tilde H_2 C_{12}^3 \to \Ind \tilde H_1 C_9^3 \to 0.$$
Therefore, $\tilde H_iC_{12}^3 = 0$ for $i\neq 2$ and 
$$\Res \tilde H_2C_{12}^3 = \tilde H_2C_{11}^3 + \Ind \tilde H_1 C_9^3.$$
Moreover, $\tilde H_2C_{12}^3$ does not contain any irreducible representations whose corresponding partitions have at most $2$ rows, but $\Res \tilde H_2C_{12}^3$ contains $(6,5)$ and $(7,4)$, thus $\tilde H_2C_{12}^3$ must contain $(6,5,1)$ and $(7,4,1)$ each with multiplicity $1$. Therefore, the sum of the restrictions of other irreducible representations in $\tilde H_2C_{12}^3$ is equal to 
$$A = \tilde H_2C_{11}^3 + \Ind \tilde H_1C_9^3 - \Res (6,5,1) - \Res (7,4,1).$$
Let $B$ be the set containing all irreducible representations that appears in the induction of $A$ whose restrictions are contained in $A$. Write down the equations that make the restriction of $B$ equal to $A$. Then we solve for the non-negative integer solutions of that system. For this problem, it is easy to see that we have a unique solution as in the statement of the proposition.
\end{proof}

\begin{Proposition} The only non-vanishing homology groups of $C_{13}^3$ are
\begin{align*}
\tilde H_3 C^3_{13} & = (9,1,1,1,1) \oplus (7,3,1,1,1) \oplus ( 5,5,1,1,1) \oplus (5,3,3,1,1)\\
			& \oplus (3,3,3,3,1)\\
\tilde H_2 C^3_{13} & = (7,5,1) \oplus (7,3,3) \oplus (6,5,2) \oplus (5,5,3).
\end{align*}
\end{Proposition} 
\begin{proof} Applying the equivariant long exact sequence (\ref{equiv}) with $|A| = 13$ we have exact sequences 
$$0 \to \Res \tilde H_i C_{13}^3 \to 0$$
for $i\neq 2,3$ and an exact sequence 
$$0 \to \Res \tilde H_3C_{13}^3 \to \Ind \tilde H_2C_{10}^3 \to \tilde H_2C_{12}^3\to  \Res \tilde H_2 C_{13}^3 \to \Ind \tilde H_1 C_{10}^3 \to 0.$$
Therefore, $\tilde H_i C_{13}^3 = 0$ for $i\neq 2, 3$. From the sequence, $\Res \tilde H_3C_{13}^3$ is contained in $\Ind \tilde H_2C_{10}^3$ and contains $\Ind \tilde H_2C_{10}^3 - \tilde H_2C_{12}^3$. There is a unique solution as indicated in the proposition. Therefore, 
$$\Res \tilde H_2C_{13}^3 = \Res \tilde H_3C_{13}^3 - \Ind \tilde H_2C_{10}^3 + \tilde H_2C_{12}^3 + \Ind \tilde H_1C_{10}^3.$$
Moreover, $\Res \tilde H_2C_{13}^3$ does not contain any irreducible representations whose corresponding partitions have at most two rows, but its restriction contains $(7,5)$, thus it must contain $(7,5,1)$ with multiplicity $1$. Therefore, the sum of the restrictions of the other irreducible representations in $\tilde H_2C_{13}^3$ is equal to 
$$(5,4,3) \oplus 2\cdot (5,5,2) \oplus (6,3,3) \oplus (6,4,2) \oplus (6,5,1) \oplus (7,3,2).$$
It is easy to see that there is a unique solution as stated in the proposition.
\end{proof}

\section{Proof of the main theorem}
In this section we determine the homology groups of the matching complexes $C^3_n$ for $14\le n \le 24$ using the equivariant long exact sequence (\ref{equiv}) and the duality as stated in Proposition \ref{duality3}. Note that to determine the homology groups of $C^3_{20}$ and $C^3_{23}$ we need to compute the dimensions of $K_{p,0}(V,2)$ for $p = 6,7$. In the following, for a representation $W$ of the symmetric group $S_N$ and a positive number $k$, we denote by $W_k$ the direct sum of all irreducible representations of $W$ whose corresponding partitions have at most $r$ rows and let $W^k = W - W_k$. 

\begin{Proposition}\label{Prop2.1} The only homology groups of $C_n^3$ for $14 \le n \le 24$ containing irreducible representations whose corresponding partitions have at most $3$ rows are $\tilde H_3 C_{14}^3$, $\tilde H_3 C_{15}^3$, $\tilde H_3 C_{16}^3$, $\tilde H_4C_{17}^3$, $\tilde H_4C_{18}^3$, $\tilde H_4C_{19}^3$, $\tilde H_4C_{20}^3$, $\tilde H_5C_{21}^3$, $\tilde H_5C_{22}^3$ and $\tilde H_5C_{23}^3$. Moreover, 
$$(\tilde H_4C_{20}^3)_3 = (8,8,4) \oplus (8,6,6), \text{ and } (\tilde H_5C_{23}^3)_3 = (9,8,6).$$
\end{Proposition}
\begin{proof}This follows from the results of the homology of matching complexes $C^3_n$ for $n \le 13$ in section \ref{Sec1} and the duality in Proposition \ref{duality3}.
\end{proof}

\begin{Proposition}\label{computation1} The only non-vanishing homology groups of $C_n^3$ for $14 \le n \le 19$ are $\tilde H_3C_{14}^3$, $\tilde H_3 C^3_{15}$, $\tilde H_3 C^3_{16}$, $\tilde H_4C_{16}^3$, $\tilde H_4C_{17}^3$, $\tilde H_4C_{18}^3$, $\tilde H_4C_{19}^3$, $\tilde H_5C_{19}^3$.
\end{Proposition}
\begin{proof} The computational proof is given in our Macaulay2 package MatchingComplex.m2.
\end{proof}

\begin{Proposition}\label{C20} The only non-vanishing homology groups of $C_{20}^3$ are $\tilde H_5C_{20}^3$ and $$\tilde H_4C_{20}^3 = (8,8,4) \oplus (8,6,6).$$
\end{Proposition}
\begin{proof} Applying the equivariant long exact sequence (\ref{equiv}) with $|A| = 20$ we have exact sequences 
$$0 \to \Res \tilde H_i C_{20}^3 \to 0$$
for $i\neq 4,5$ and an exact sequence 
$$0 \to \tilde H_5C_{19}^3 \to \Res \tilde H_5C_{20}^3 \to \Ind \tilde H_4C_{17}^3 \to \tilde H_4C_{19}^3\to  \Res \tilde H_4 C_{20}^3\to 0.$$
Therefore, $\tilde H_i C_{20}^3 = 0$ for $i\neq 4, 5$. By Theorem \ref{trans}, $K_{4,2}(2) = 0$. To determine $\tilde H_4C_{20}^3$, note that by Proposition \ref{Prop2.1} and Theorem \ref{trans}, $K_{5,1}(V,2)$ contains $M = S_\l(V) \oplus S_\mu (V)$ where $\l = (8,8,4)$ and $\mu = (8,6,6)$. Moreover, using Macaulay2 to compute the dimensions of minimal linear syzygies of the module $\oplus_{k=0}^\infty \Sym^{3k+2} (V)$ with $\dim V = 4$, we get $\dim K_{6,0}(V,2) = 14003$. 
Since $\oplus_{k=0}^\infty \Sym^{3k+2} (V)$ is a Cohen-Macaulay module of codimension $16$ with $h$-vector $(10,16,1)$, 
$$\dim K_{5,1}(V,2) = 14003 - 10 \cdot \binom{16}{6} + 16\cdot \binom{16}{5} - \binom{16}{4} = 1991.$$
Since $\dim M =  1991$, $K_{5,1}(V,2) \cong M$. By Theorem \ref{trans}, 
$$\tilde H_4C_{20}^3 = (8,8,4) \oplus (8,6,6).$$
Finally, 
$$\Res \tilde H_5C_{20}^3 = \tilde H_5C_{19}^3 + \Ind \tilde H_4C_{17}^3 - \tilde H_4C_{19}^3 +  \Res \tilde H_4 C_{20}^3.$$
This determines $\tilde H_5C_{20}^3$ as given in our package MatchingComplex.m2.
\end{proof}

\begin{Proposition}\label{computation2} The only non-vanishing homology groups of $C_{21}^3$ and $C^3_{22}$ are $\tilde H_5C_{21}^3$, $\tilde H_5C_{22}^3$, $\tilde H_6C_ {22}^3$.
\end{Proposition}
\begin{proof} The computational proof is given in our Macaulay2 package MatchingComplex.m2.
\end{proof}

\begin{Proposition}\label{C23} The only non-vanishing homology groups of $C^3_{23}$ are $\tilde H_6C_{23}^3$ and
\begin{align*}
\tilde H_5C_{23}^3 &= (9, 8, 6)  \oplus (8, 6, 6, 3) \oplus (8, 7, 6, 2) \oplus (8, 8, 4, 3) \oplus (8, 8, 5, 2)\\
		& \oplus (8, 8, 6, 1) \oplus (9, 6, 6, 2) \oplus (9, 7, 6, 1)  \oplus (9, 8, 4, 2) \oplus (9, 8, 5, 1)\\
		& \oplus (10, 6, 6, 1) \oplus (10, 8, 4, 1).
\end{align*}
\end{Proposition}
\begin{proof} Applying the equivariant long exact sequence (\ref{equiv}) with $|A| = 23$ we have exact sequences 
$$0 \to \Res \tilde H_i C_{23}^3 \to 0$$
for $i\neq 5,6$, and an exact sequence
\begin{align}
  \label{Eq1}
  \begin{split}
    0 \to \tilde H_6C_{22}^3 \to \Res \tilde H_6C_{23}^3 &\to \Ind \tilde H_5C_{20}^3 \to \tilde H_5C_{22}^3\to\\
    &\to \Res \tilde H_5C_{23}^3 \to \Ind \tilde H_4C_{20}^3 \to 0.
  \end{split}
\end{align}

Therefore, $\tilde H_iC^3_{23} = 0$ for $i\neq 5,6$. Moreover, by Proposition \ref{C20}, $(\Ind \tilde H_5C_{20}^3)_3 = 0$. Therefore, we have an exact sequence
\begin{equation}\label{Eq2}
0 \to (\tilde H_5C_{22}^3)_3 \to (\Res \tilde H_5C_{23}^3)_3 \to (\Ind \tilde H_4C_{20}^3)_3 \to 0.
\end{equation}
By Proposition \ref{Prop2.1}, we know that 
$$(\tilde H_5C_{23}^3)_3 = (9,8,6).$$
Let $Y$ be the direct sum of irreducible representations whose corresponding partitions have $4$ rows in $\tilde H_5C_{23}^3$. Then from exact sequence (\ref{Eq2}), we have
\begin{align*}
(\Res Y)_3 &= (\tilde H_5C_{22}^3)_3 + (\Ind \tilde H_5C_{20}^3)_3 - \Res (9,8,6)\\
	& = (8, 8, 6) \oplus (9, 7, 6) \oplus (9, 8, 5)\oplus (10, 6, 6)\oplus (10, 8, 4).
\end{align*}
Therefore, 
$$Y_1  = (8, 8, 6,1) \oplus (9, 7, 6,1) \oplus (9, 8, 5,1)\oplus (10, 6, 6,1)\oplus (10, 8, 4,1)$$
is a subrepresentation of $Y$. By Proposition \ref{C20}, $\tilde H_4C_{20}^3 = (8,8,4) \oplus (8,6,6)$, thus $\Res Y_1 \oplus \Res (9,8,6) $ maps surjectively onto $\Ind \tilde H_4C_{20}^3$. Let 
$$D = \tilde H_5C_{22}^3 + \Ind \tilde H_4C_{20}^3 - \Res Y_1 - \Res (9,8,6).$$
Then the exact sequence (\ref{Eq1}) becomes 
\begin{equation}\label{Eq3}
0 \to \tilde H_6C_{22}^3 \to \Res \tilde H_6C_{23}^3 \to \Ind \tilde H_5C_{20}^3 \to D \to \Res Z \to 0
\end{equation}
where $Z = \tilde H_5C_{23}^3 - Y_1 - (9,8,6)$. 

By Proposition \ref{computation2}, $\tilde H_6C_{22}^3$ contains only irreducible representations whose corresponding partitions have $8$ rows while $\tilde H_5C_{22}^3$ contains only irreducible representations whose corresponding partitions have at most $6$ rows. Thus we have 
$$0 \to \tilde H_6C_{22}^3 \to (\Res \tilde H_6C_{23}^3)^6 \to (\Ind \tilde H_5C_{20}^3)^6\to 0$$
where $(\Res \tilde H_6C_{23}^3)^6  = \Res \tilde H_6C_{23}^3 - (\Res \tilde H_6C_{23}^3)_6$ and $(\Ind \tilde H_5C_{20}^3)^6 =\Ind \tilde H_5C_{20}^3  - (\Ind \tilde H_5C_{20}^3)_6.$ This exact sequence determines the irreducible representations of $\tilde H_6C_{23}^3$ whose corresponding partitions have $7$ or $8$ rows as given in our package MatchingComplex.m2. Let 
$$C = \Ind \tilde H_5C_{20}^3 + \tilde H_6C_{22}^3 - \Res (\tilde H_6C_{23}^3)^6,$$
where $(\tilde H_6C_{23}^3)^6$ is the direct sum of irreducible representations in $\tilde H_6C_{23}^3$ whose corresponding partitions have $7$ and $8$ rows described above. Then the exact sequence (\ref{Eq3}) becomes
\begin{equation}\label{Eq4}
0 \to \Res (\tilde H_6C_{23}^3)_6 \to C \to D \to \Res Z \to 0.
\end{equation}
Finally, using Macaulay2 to compute the dimensions of minimal linear syzygies of the module $\oplus_{k=0}^\infty \Sym^{3k+2} (V)$ with $\dim V = 4$, we get $\dim K_{7,0}(V,2) = 5400$. Moreover, $\tilde H_4C_{23}^3 = 0$, thus by Theorem \ref{trans}, $K_{5,2}(V,2) = 0$. Since $\oplus_{k=0}^\infty \Sym^{3k+2} (V)$ is a Cohen-Macaulay module of codimension $16$ with $h$-vector $(10,16,1)$, 
$$\dim K_{6,1}(V,2) = 10 \cdot \binom{16}{7} - 16\cdot \binom{16}{6} + \binom{16}{5} - 5400 = 14760.$$
Since the sum of dimensions of irreducible representations $S_\l(V)$ corresponding to partitions $\l$ in $Y_1\oplus(9,8,6)$ is equal to $11520$, the sum of dimensions of irreducible representations $S_\l(V)$ corresponding to partitions $\l$ in $Z$ is $14760 - 11520 = 3240.$ This fact and exact sequence (\ref{Eq4}) determine $\tilde H_5C_{23}^3$ as stated in the proposition and $\tilde H_6C_{23}^3$ as given in our package MatchingComplex.m2.
\end{proof}

\begin{Theorem}\label{main} The third Veronese embeddings of projective spaces satisfy property $N_6$.
\end{Theorem}
\begin{proof} It remains to prove that the only non-zero homology groups of the matching complex $C^3_{24}$ is $\tilde H_6C_{24}^3$. Applying the equivariant long exact sequence (\ref{equiv}) with $|A| = 24$ we have exact sequences
$$0 \to \Res \tilde H_i C_{24}^3 \to 0$$
for $i\neq 5,6$, and an exact sequence
$$0 \to \tilde H_6C_{23}^3 \to \Res \tilde H_6C_{24}^3 \to \Ind \tilde H_5C_{21}^3 \to \tilde H_5C_{23}^3 \to \Res \tilde H_5C_{24}^3 \to 0.$$
Therefore, $\tilde H_iC_{24}^3 = 0$ for $i\neq 5,6$ and $\tilde H_5C_{23}^3$ maps surjectively onto $\Res \tilde H_5C_{24}^3$. Moreover, by the result of Ottaviani-Paoletti \cite{OP}, the third Veronese embedding of $\PP^3$ satisfies property $N_6$. By Theorem \ref{trans}, $\tilde H_5C_{24}^3$ does not contain any irreducible representations whose corresponding partitions have at most $4$ rows. By Proposition \ref{C23}, $\tilde H_5C_{23}^3$ contains only irreducible representations whose corresponding partitions have at most $4$ rows, thus $\tilde H_5C_{24}^3$ is zero.
\end{proof}
\section*{Appendix} 
In this appendix we explain the ideas behind our Macaulay2 package MatchingComplex.m2. To determine the homology groups of $C^3_N$ inductively we use the equivariant long exact sequence (\ref{equiv}) with $|A| = N$ to determine two representations $B$ and $C$ of $S_{N-1}$ satisfying the following property. $B$ is a subrepresentation of $\Res \tilde H_iC^3_N$ and $C$ is a superrepresentation of $\Res \tilde H_iC^3_N$. To determine $\tilde H_iC^3_N$, we need to find a representation $X$ of $S_N$ satifying the property that $\Res X$ is a subrepresentation of $C$ and a superrepresentation of $B$. The function findEquation in our package will first determine a list $D$ of all possible partition $\l$ of $N$ such that $\Res \l$ is a subrepresentation of $C$. Let $X = \sum_{\l \in D} x_\l \cdot \l$ where $x_\l \ge 0$ is the multiplicity of the partition $\l\in D$ that we need to determine. Restricting $X$ we have inequalities (obtained by calling findEquation B and findEquation C) expressing the fact that $\Res \tilde H_iC^3_N$ is a subrepresentation of $C$ and a superrepresentation of $B$. We then use maple to solve for non-negative integer solutions of this system of inequalities.

\section*{Acknowledgements}
I would like to thank Claudiu Raicu for introducing the problem and its relation to the homology of matching complexes and my advisor David Eisenbud for useful conversations and comments on earlier drafts of the paper.

\bibliographystyle{amsalpha}

\end{document}